\documentclass[12pt,reqno]{amsart}
\usepackage{amsmath}
\usepackage{graphics}
\usepackage{amssymb}
\newtheorem{thm}{Theorem}

\newcommand{\vertiii}[1]{{\left\vert\kern-0.25ex\left\vert\kern-0.25ex\left\vert
#1 \right\vert\kern-0.25ex\right\vert\kern-0.25ex\right\vert}}

\def \re   {\text {\rm Re}\, }
\def \diag  {\text {\rm diag}}
\def \cos   {\text {\rm cos}}

\def \min   {\text {\rm min}}
\def \Re    {\text {\rm Re}}

\def \diag  {\text {\rm diag}}
\def \lim   {\text {\rm lim}}

\def \tr    {\text {\rm tr}}

\def \argmin {\text {\rm argmin}}

\begin{document}
\title[]{On the Bures-Wasserstein Distance Between Positive
Definite Matrices}

\author[Rajendra Bhatia]{Rajendra Bhatia}
\address{Ashoka University, Sonepat,\\ Haryana, 131029, India}

\email{rajendra.bhatia@ashoka.edu.in}

\author[Tanvi Jain]{Tanvi Jain}
\address{Indian Statistical Institute\\ New Delhi 110016, India}
\email{tanvi@isid.ac.in}

\author[Yongdo Lim]{Yongdo Lim}

\address{Department of Mathematics, Sungkyunkwan University\\ Suwon 440-746, Korea}

\email{ylim@skku.edu}



\begin{abstract}
The metric $d(A,B)=\left[ \tr\, 
A+\tr\, B-2\tr(A^{1/2}BA^{1/2})^{1/2}\right]^{1/2}$ on the manifold
of $n\times n$ positive definite matrices arises in various optimisation problems, in quantum information and in the
theory of optimal transport. It is also related to Riemannian geometry.
In the first part of this paper we study this metric from the perspective of 
matrix analysis, simplifying and unifying various proofs. Then we develop a 
theory of a mean of two, and a barycentre of several,
positive definite matrices with respect to this metric. We explain some recent 
work on a fixed point iteration for computing this Wasserstein barycentre. Our 
emphasis is on ideas natural to matrix analysis.
\end{abstract}

\maketitle

\section{Introduction}
Let $\mathbb{M} (n)$ be the space of  $n\times n$ complex matrices,
$\mathbb{H}(n)$ the real subspace of $\mathbb{M} (n)$ consisting of
Hermitian matrices, and $\mathbb{P}(n)$ the subset of
$\mathbb{H}(n)$ consisting of positive semi definite (psd) matrices.
The Frobenius inner product on $\mathbb{M}(n)$ is defined as
$\langle A,B \rangle = \re\tr A^*B,$  and the associated norm
 $\|A\|_2 = (\tr A^*A)^{1/2}$ is called the Frobenius norm. Every psd matrix
$A$
has a unique psd square root, which we denote by $A^{1/2}.$ Given $ A,B$ in
$\mathbb{P} (n)$ define $d(A,B)$ by the relation
\begin{equation}
d(A,B) = \left [\tr\,A+\tr\,B - 2\tr\,\left(A^{1/2}BA^{1/2}
\right)^{1/2}\right ]^{1/2}.\label{eq1}
\end{equation}
It turns out that $d(A,B)$ is a metric on the space $\mathbb{P}(n)$. This metric
has been of interest in quantum information where it is called the {\it Bures distance,} and in statistics
and the theory of optimal transport where it is called the {\it Wasserstein metric.}
If $A$ and $B$ are diagonal matrices, then $d(A,B)$ reduces to the Hellinger distance between probability distributions and is related to the Rao-Fisher metric in information theory.
The metric $d$ is of interest in differential geometry, as it is the distance function corresponding to a Riemannian metric.

In this paper we explore some fundamental properties of this metric from the perspective of matrix analysis.
This allows us to unify several known facts and to simplify their proofs,
to point out new connections, to raise new questions and to answer some of them.

\section{Some variational principles}

The metric $d(A,B)$ and the quantity $(A^{1/2}BA^{1/2})^{1/2}$ occurring in it, both are
related to solutions of extremal problems arising in different contexts.

Recall that a matrix $A$ is psd if and only if it can be expressed as $A=MM^*$
for some $M \in \mathbb {M}(n).$ Another matrix $N$ satisfies the relation
$A=NN^*$ if and only if $ N=MU$ for some unitary matrix $U.$ One special matrix
among all these is $A^{1/2}.$ Let $U(n)$ stand for the group of all unitary
matrices. Given a psd matrix $A$ let $\mathcal{F}(A)$ be the set defined as
\begin{eqnarray}
\mathcal{F}(A) &=& \left \{M \in \mathbb{M}(n) : A = MM^*\right \} \nonumber\\
&=&\left \{A^{1/2} U :  U \in U (n)\right \}. \label{eq2}
\end{eqnarray}
\begin{thm}\label{thm1}
If $d(A,B)$ is defined as in \eqref{eq1}, then\\
\begin{eqnarray}
d(A,B) &=& \quad \underset{M \in\mathcal{F}(A) \atop {N \in
\mathcal{F}(B)}}{\min} \quad \|M-N\|_2 \nonumber \\
&=& \quad \underset{U \in U(n)}{\min} \,\,\,\,\, \|A^{1/2} -
B^{1/2}U\|_2. \label{eq3}
\end{eqnarray}
The minimum in the last expression in \eqref{eq3} is attained at a
unitary $U$ occurring in the polar decomposition of
$B^{1/2}A^{1/2}:$
\begin{equation*}
B^{1/2}A^{1/2}=U|B^{1/2}A^{1/2}|=U(A^{1/2}BA^{1/2})^{1/2}.
\end{equation*}
\end{thm}

\begin{proof}
We have for every $U \in U(n)$
\begin{eqnarray*}
\|A^{1/2} - B^{1/2} U\|_{2}^{2} &=& \|A^{1/2} \|_{2}^{2} + \|B^{1/2}
\|_{2}^{2}  -\tr (A^{1/2} U^* B^{1/2} +A^{1/2} B^{1/2} U)\\
&=& \tr A + \tr B - \tr (U^*B^{1/2}A^{1/2} +  A^{1/2} B^{1/2} U).
\end{eqnarray*}
Hence,
\begin{eqnarray}
 \lefteqn{\underset{U \in U(n)}{\min} \,\,\,\, \|A^{1/2}-B^{1/2} U\|^2_2}
\nonumber\\
&=& \tr A + \tr B - \underset{U \in U(n)}{\max} \tr (U^*B^{1/2}A^{1/2} +
A^{1/2} B^{1/2} U). \label{eq4}
\end{eqnarray}
To evaluate the maximum in \eqref{eq4} let $X =B^{1/2}A^{1/2}.$ Then $|X|:=
(X^*X)^{1/2} = (A^{1/2}BA^{1/2})^{1/2}.$ Let $X=VP$ be the polar
decomposition of $X,$ where $P=|X|$ and $V$ is unitary. Then

$$\tr(U^*B^{1/2}A^{1/2} +  A^{1/2} B^{1/2} U) =\tr(U^*X + X^* U) =  \tr(U^*VP +
PV^* U).$$

Putting $W=U^*V,$ the last expression above can be written as $\tr (W + W^*)
P.$ Choosing a basis in which $W =
\diag(e^{i\theta_{1}},\cdots,e^{i\theta_{n}})$ we have
$$\tr(W+W^*)P = \sum\limits^{n}_{j=1}(2 \, \cos\,\theta_j) p_{jj}.$$
The maximum value of this is attained when $W=I$ and is equal to
$$\sum\limits^{n}_{j=1} 2p_{jj} = 2 \tr P = 2 \tr |X|.$$
So, from \eqref{eq4} we have
$$ \underset{U \in U(n)}{\min} \|A^{1/2}-B^{1/2} U\|^2_2
=\tr A + \tr B - 2\tr (A^{1/2}B A^{1/2})^{1/2}.$$ This
shows the equality of the two extreme sides of \eqref{eq3}. The
expression in the middle is equal to this because of the unitary
invariance of $\|\cdot\|_2.$ In the course of the proof we saw that
the minimum in \eqref{eq3} is attained when $W=U^*V=I$; i.e., when
$U=V,$ the polar factor for $X=B^{1/2}A^{1/2}.$
\end{proof}

From the representations in \eqref{eq3} it is easy to see that $d(A,B)$ is
indeed a metric. Obviously $d(A,B)\geq 0.$ The compact sets $\mathcal{F}(A)$
and $\mathcal{F}(B)$ are disjoint unless $A=B.$ So $d(A,B)=0$ if and only if
$A=B.$
To prove the triangle inequality, note that for all psd matrices $A,$ $B,$ $C,$ and unitaries $U,$ $V,$ we have
\begin{eqnarray*}
d(A,B) & \le & \|A^{1/2}-B^{1/2}U\|_2\\
 & \le & \|A^{1/2}-C^{1/2}V\|_2+\|B^{1/2}U-C^{1/2}V\|_2\\
 & = & \|A^{1/2}-C^{1/2}V\|_2+\|B^{1/2}-C^{1/2}VU^*\|_2.
\end{eqnarray*}
Taking the minimum over all $U,$ $V,$ we see that $d(A,B) \leq d(A,C) + d (B,C).$
This proof is adopted from \cite{hay}.

A well-known and important problem in factor analysis and in
multidimensional scaling is the orthogonal Procrustes problem. This
asks for the solution of the minimisation problem $\min\,
\|A-BU\|_2$ where $A$ and $B$ are given matrices (not necessarily
psd) and $U$ varies over unitaries. The argument in Theorem
\ref{thm1} shows that the minimum is attained when $U$ is the
unitary polar factor of $B^*A.$ In applications $A$ and $B$
represent multivariate data sets, and the problem is to ascertain
whether they are equivalent up to a rotation. See \cite{hig} for a
brief and \cite{go} for an expansive discussion.

In the following remarks we point out some more connections between the
Bures distance and some other classical problems in matrix analysis.

\begin{enumerate}
 \item[1.]  The expression in \eqref{eq1} is reminiscent of the matrix
arithmetic-geometric mean inequality \cite{bhc}. Indeed this
inequality tells us that for any two psd matrices $\vertiii{A^{1/2}
B^{1/2}}\leq \frac{1}{2} \vertiii{A + B }$ for every unitarily
invariant norm. For the trace norm $\|\cdot \|_1$ the left hand side
of this inequality is equal to $\tr (A^{1/2} BA^{1/2})^{1/2}$ and
the right hand side to $\frac{1}{2} (\tr A + \tr B). $ That these
two quantities are equal if and only if $A=B$ is one of the
assertions included in the statement that $d(A,B)$ is a metric.
(This has been known for the Schatten
p-norms, $1<p<\infty$ \cite{k}, and is false for the case $p = \infty.$)\\
\item[2.] Let $p = (p_1,\ldots, p_n)$ and $q = (q_1,\ldots, q_n)$ be
nonnegative vectors, and let 
\begin{equation}
\rho(p,q) = \left[ \sum_{i=1}^{n} \left(\sqrt{p_i}-\sqrt{q_i}\right)^2
\right ]^{1/2}. \label{eq5}
\end{equation}
This is the $l_2$ norm distance between the square roots of the vectors $p$ and $q.$
If $ p$ and $q$ are probability distributions (i.e., $\sum p_i = \sum q_i =
1$), then $\rho(p,q)$
is called the {\it Hellinger distance}.
In analogy, one could define a distance $\rho(A,B)$ on psd matrices by putting $\rho(A,B)=\|A^{1/2}-B^{1/2}\|_2.$ When $A$ and $B$ commute, the distance $d(A,B)$ is equal to $\rho(A,B).$\\

 \item[3.] A psd matrix $A$ with $ \tr A =1$ is called a {\it density matrix}
or a {\it state}. The quantity
\begin{equation}
 F(A,B) = \tr (A^{1/2} BA^{1/2})^{1/2} \label{eq6}
\end{equation}
is called the {\it fidelity} between two states $A$ and $B.$ In this case from
\eqref{eq1} we see that
\begin{equation}
\frac{1}{2} d^2(A,B) = 1-F(A,B). \label{eq7}
\end{equation}
In the quantum information theory literature it is customary to define the Bures distance between density matrices $A$ and $B$
as the quantity $\sqrt{1-F(A,B)}.$
This is just the distance \eqref{eq1} restricted to density matrices.
An illuminating discussion of the Bures distance from the QIT perspective can be found in \cite{bz}.
If $A=uu^*$ for some unit vector $u,$ then $A$ is called a {\it pure state}. In
this case we have
$$F(A,B) = \langle u, Bu\rangle^{1/2}.$$
If both $A$ and $B$ are pure states given as $A=uu^*,$ $B=vv^*,$ then\\
$$F(A,B) =| \langle u, v\rangle|,$$
and
$$ \frac{1}{2} d^2(A,B) = 1-|\langle u, v\rangle|. $$
\item[4.] The Bures distance is related to a measure of separation between subspaces of $\mathbb{C}^n.$ Let $\mathcal{M}$ and $\mathcal{N}$ be two $l$-dimensional subspaces of $\mathbb{C}^n,$ and let $P$ and $Q$ be the orthogonal projections with ranges $\mathcal{M}$ and $\mathcal{N},$ respectively. Among all unitary operators on $\mathbb{C}^n$ that map $\mathcal{M}$ onto $\mathcal{N},$ there is a special one called a {\it direct rotation}. This unitary operator $U$ can be represented in a particular orthonormal basis as
\begin{equation*}
U=\begin{bmatrix}C & -S & O\\
S & C & O\\
O & O & I\end{bmatrix},
\end{equation*}
where $C$ and $S$ are nonnegative diagonal matrices. If $2l\le n$, then $C$ and $S$ are $l\times l$ matrices, and if $2l>n,$ then they are $(n-l)\times (n-l)$
matrices. Further, $C^2+S^2=I.$ The operator
$\Theta(\mathcal{M},\mathcal{N})=\arccos\, C$ is called the {\it angle operator} between $\mathcal{M}$ and $\mathcal{N}.$
 The diagonal entries of this diagonal operator are called the {\it canonical angles} between the spaces $\mathcal{M}$ and $\mathcal{N}.$
 It can be seen that the nonzero singular values of $PQ$ are the nonzero diagonal entries of $C.$ The direct rotation was used in \cite{dk}
in connection with perturbation theory of eigenvectors. See also \cite{rbh}, Section VII.1 and \cite{s} Chapter II, Section 4.
The fidelity between projections $P$ and $Q$ is the sum of the cosines of the canonical angles between the spaces $\mathcal{M}$ and $\mathcal{N}$:
\begin{equation*}
F(P,Q)=\tr (PQP)^{1/2}=\|PQ\|_1=\sum c_j.
\end{equation*}
Here $c_j$ are the diagonal entries of $C$ if $2l\le n.$ In the case when $2l>n,$ we take $c_j$ to be the diagonal entries of $C$ for $1\le j\le n-l$ and
 take them to be $0$ for $n-l<j\le l.$
They are thus the cosines of the canonical angles between $\mathcal{M}$ and $\mathcal{N}.$
We have
\begin{equation*}
\frac{1}{2} d^2 (P,Q) = \sum (1-c_j).
\end{equation*}

\end{enumerate}

The fidelity $F(A,B)$ is a quantity of great interest and it is useful to have
more descriptions of it. Some variational characterisations of it are given
below. We need some facts from the theory of geometric means. See Chapter 4 of
\cite{rbh1}.\\ Let $A$ and $B$ be positive definite matrices. Their geometric
mean $A\#B$ is defined by the Pusz-Woronowicz formula \cite{pw}
\begin{equation} A\#B =
A^{1/2} (A^{-1/2} B A^{-1/2})^{1/2} A^{1/2}. \label{eq8} \end{equation} This mean
is symmetric in $A$ and $B.$ It is the unique positive definite solution of the
Riccati equation \begin{equation}  X A^{-1}X = B. \label{eq9} \end{equation} The
matrix $AB$ has positive eigenvalues, and it has a unique square root
$(AB)^{1/2}$  that has positive eigenvalues. The eigenvalues of $BA$ are the
same as those of $AB.$ We have \begin{equation} A\#B = A(A^{-1} B)^{1/2} =
(AB^{-1})^{1/2}B. \label{eq10} \end{equation} Another useful characterisation is

\begin{equation} A\#B = \max \left\{X: \left[\begin{array}{cc} A & X \\ X &
B \end{array} \right] \geq 0\right\}. \label{eq11} \end{equation}
Here the
maximum is with respect to the Loewner partial order; for Hermitian matrices $X$
and $Y$ we say $X \geq Y$ if $ X - Y$ is psd. We recall also two necessary and
sufficient conditions for the block matrix \begin{equation}
\left[\begin{array}{cc}A & X \\ X^* & B\end{array}  \right]\label{eq12}
\end{equation} to be psd. The first says that the matrix \eqref{eq12} is psd if
and only if \begin{equation} A \geq X B^{-1} X^*, \label{eq13} \end{equation}
and the second that this is so if and only if there exists a contraction $K$ (an operator with $\|K\|\le 1$)
such that \begin{equation} X = A^{1/2}KB^{1/2}. \label{eq14} \end{equation} See
Chapter 1 of \cite{rbh1}.\\

\begin{thm}\label{thm2}
Let $A$ and $B$ be positive definite matrices.
Then
\begin{eqnarray}
{\rm (i)} \qquad F(A,B) &=& \underset{X>0}{\min}\,\,\, \frac{1}{2} \tr (AX +
BX^{-1}). \label{eq15}\\
{\rm (ii)} \qquad F(A,B) &=& \underset{X>0}{\min}\,\,\,  \sqrt{ \tr (AX)
\tr (BX^{-1})}. \label{eq16}\\
{\rm (iii)} \qquad F(A,B) &=& \underset{X>0}{\max}\left\{|\tr X| : A \geq X B^{-1}X^*
\right\}. \label{eq17}
\end{eqnarray}
\end{thm}

\begin{proof}
 (i) Consider the function $f(X) =\tr(AX + BX^{-1})$ defined on
$\mathbb{P}(n).$ This is a convex function and its derivative $Df(X)$ is the linear map from
$\mathbb{H}(n)$ into $\mathbb{R}$ given by the formula.
$$Df(X)(Y) =\tr(AY-BX^{-1} YX^{-1}), \,\,\, Y \in \mathbb{H}(n).$$\\
(See \cite{rbh} pp.310 - 312.) So a point $X_0$ is a minimum for $f$ if and only if \\
$$\tr(A - X_0^{-1}B X_0^{-1})Y=0, \,\,\, Y \in \mathbb{H}(n).$$\\
This is so if and only if $A-X_0^{-1} B X_0^{-1} = 0,$ or in other words $
X_0AX_0 = B.$ This is the Riccati equation \eqref{eq9}. So, $X_0 = A^{-1} \#
B.$ We have then $$\underset{X>0}{\min} f(X) = f(X_0)=\tr(A(A^{-1}\#B) +
B(A\#B^{-1})). $$
Using \eqref{eq10}, the right hand side of this equation can be expressed as
\begin{eqnarray*}
\tr(A\cdot A^{-1}(AB)^{1/2} + B(AB)^{1/2} B^{-1}) &=& 2 \tr(AB)^{1/2} \\
&=& 2\tr(A^{1/2} BA^{1/2})^{1/2} = 2F(A,B).
\end{eqnarray*}
This proves (i).\\
(ii) In the proof of (i) above we have seen that at $X_0 = A^{-1} \#B,$ we have
$$\tr AX_0 = \tr BX_0^{-1} = F(A,B).$$ So
$$ \frac{\tr AX_0 +\tr BX_0^{-1} }{2} =\sqrt{\tr (AX_0)  \tr (BX_0^{-1})}.$$\\
This proves (ii).\\
(iii) We have remarked earlier that\\
$$ A\geq MB^{-1} M^* \Leftrightarrow \left[\begin{array}{cc}A & M \\ M^* &
B\end{array} \right] \geq 0 \Leftrightarrow M=A^{1/2} KB^{1/2}$$
for some contraction $K.$ By the Schwarz inequality we have
\begin{eqnarray*}
 |\tr M |=| \tr (A^{1/2} KB^{1/2}| &\leq& \|A^{1/2}
K\|_{2}\|B^{1/2}\|_2\\
&\leq& \|A^{1/2}\|_2 \| B^{1/2}\|_2 =\sqrt{\tr A \tr B}.
\end{eqnarray*}
If $$\left[\begin{array}{cc} A & M \\ M^* & B \end{array} \right]\geq 0 ,$$ then
for every $X > 0$ we have\\
\begin{eqnarray*}
 0 &\leq& \left[\begin{array}{cc} X^{1/2} & O \\ O & X^{-1/2} \end{array}
\right] \left[\begin{array}{cc} A & M \\ M^* & B \end{array}
\right] \left[\begin{array}{cc} X^{1/2} & O \\ O & X^{-1/2} \end{array}
\right]\\
& = &\left[\begin{array}{cc} X^{1/2} AX^{1/2} & X^{1/2} MX^{-1/2}  \\
X^{-1/2} M^*X^{1/2} & X^{-1/2} BX^{-1/2} \end{array} \right].
\end{eqnarray*}
Hence
$$ |\tr X^{1/2} MX^{-1/2} | \leq \sqrt{\tr(X^{1/2} A X^{1/2})\,\,\, \tr
(X^{-1/2} B X^{-1/2})}. $$ \\
In other words,
$$ |\tr M| \leq \sqrt{\tr(A X)\,\,\,\tr (BX^{-1})}. $$
This is true for all $M$ satisfying the condition $ A\geq MB^{-1}M^* $ and for
all $ X>0.$ So
\begin{eqnarray}
 \max \left \{|\tr M| : A \geq MB^{-1} M^* \right\} & \leq & \underset{X>0}{\min}
\sqrt{\tr(AX) \tr(BX^{-1})} \nonumber \\
&=& F(A,B). \label{eq18}
\end{eqnarray}
Let $M=(AB)^{1/2} = A(A^{-1} \# B).$ Then
\begin{eqnarray*}
MB^{-1} M^* &=& (AB)^{1/2} B^{-1} (BA)^{1/2} = B^{-1}B(AB)^{1/2} 
B^{-1}(BA)^{1/2}\\
&=& B^{-1}(BA)^{1/2}(BA)^{1/2}= B^{-1} BA = A.
\end{eqnarray*}
So, the maximum on the left hand side of \eqref{eq18} is attained when $
M=(AB)^{1/2}$ and it is equal to $\tr (A^{1/2} B A^{1/2})^{1/2} = F(A,B).$  This
proves (iii).
\end{proof}

Theorem \ref{thm2} with different proofs can be found in \cite{al,u2}.

\section{The Statistical distance}
Let $X,Y$ be complete separable metric spaces and let $\mu,\nu$ be Borel probability measures
on $X$ and $Y,$ respectively.
Let $\Gamma(\mu,\nu)$ be the collection of probability
measures $\gamma$ on $X\times Y$ whose marginals
are $\mu$ and $\nu.$
Let $c(x,y)$ be a nonnegative Borel measurable function on $X\times Y.$
The {\it optimal transport problem} is the 
minimisation problem of finding
\begin{equation*}
\inf\limits_{\gamma\in\Gamma(\mu,\nu)}\int\limits_{X\times Y}c(x,y)d\gamma(x,y).
\end{equation*}
(Here $\mu,\nu$ are thought of as mass distributions,
and $c(x,y)$ is the cost of moving a unit mass from $X$ to $Y.$
The problem is of moving one mass distribution to another at the least cost.)\\

An important special case of this problem is the following.
Let $X=Y=\mathbb{C}^n,$ let $\mu,\nu$
have finite second moments,
and let $c(x,y)=\|x-y\|^2.$
In this case it can be shown that the quantity
\begin{equation}
d_{W}(\mu,\nu)=\inf\limits_{\gamma\in\Gamma(\mu,\nu)}\left(\int\limits_{\mathbb{R}^n\times\mathbb{R}^n}\|x-y\|^2\, d\gamma(x,y)\right)^{1/2}\label{eqrr19}
\end{equation}
defines a metric, which is called the {\it $2$-Wasserstein distance} between $\mu$ and $\nu.$
The integral on the right hand side of \eqref{eqrr19} is also written as
$E\|x-y\|^2,$ where $E$ stands for expectation.\\

In the most important special case of Gaussian measures, the distance
$d_{W}$ coincides with the Bures distance,
and this is explained below.\\
 
Let $x$ and $y$ be random vectors with values in $\mathbb{C}^n,$ each having
zero mean, and with covariance matrices $A$ and $B,$ respectively. This last
statement means that
\begin{equation}
A = \left[E(\overline{x}_i x_j)\right], \,\,\,\,\ B = \left[E(\overline{y}_i
y_j)\right].
\label{eq20}
\end{equation}
We want to find $x$ and $y$ for which $E\|x-y\|^2$ is minimal.\\

The covariance matrix of the vector $(x,y)$ is
\begin{equation}
 C = \left[\begin{array}{cc} \left[E(\overline{x}_i x_j) \right]
&\left[E(\overline{x}_i y_j)\right]\\
&\\
\left[E(\overline{y}_i x_j) \right] &\left[E(\overline{y}_i y_j)\right]
\end{array} \right] = \left[\begin{array}{cc}A & M \\ M^* & B
\end{array}\right]. \label{eq21}
\end{equation}
Our problem is to minimise
\begin{eqnarray}
 E\|x-y\|^2 &=& E\left(\sum_{i=1}^{n} (|x_i|^2 + |y_i|^2 -2 \Re \,\, 
\overline{x}_i y_i)\right) \nonumber\\
&=& \sum_{i=1}^{n} E(|x_i|^2 + |y_i|^2 -2 \Re \,\, 
\overline{x}_i y_i) \nonumber\\
&=& \tr (A + B) - 2 \Re \ (\tr \,\, M). \label{eq22}
\end{eqnarray}
This is the problem of finding
\begin{equation}
\max \left \{|\tr M| : C = \left[\begin{array}{cc} A & M \\ M^* & B
\end{array}
\right] \geq 0 \right\}. \label{eq23}
\end{equation}
(As we vary $x$ and $y$ over all vectors with covariance matrices
$A$ and $B,$ the covariance matrix of $(x,y)$ varies over all psd
matrices of the form in \eqref{eq21}.) By Theorem \ref{thm2}(iii)
the value of the maximum in \eqref{eq23} is $F(A,B).$ So
\begin{eqnarray*}
\min \ E \|x-y\|^2 &=& \tr (A+B) - 2 \tr (A^{1/2} BA^{1/2})^{1/2}\\
&=& d^2 (A,B)
\end{eqnarray*}
where $d(A,B)$ is as defined in \eqref{eq1}.\\

Let $x$ be a vector with mean $0$ and covariance matrix $A.$ Then for any $T 
\in \mathbb{M}(n)$ we have 
\begin{eqnarray*}
E (\langle x, Tx \rangle) &=& \ E \left( \sum_{i,j} \,\, t_{ij} \,\, 
\overline{x_i} 
\,\, x_j \right) = \sum_{i,j} \,\, t_{ij} \,\, E (\overline{x_i} \,\, x_j ) \\
&=& \sum_{i,j} \,\, t_{ij} \,\, a_{ij} = \tr \,\, TA.
\end{eqnarray*}
Hence,
\begin{eqnarray*}
 E\|x- Tx\|^2 &=& E (\|x\|^2 + \| Tx\|^2 \,\,- \,\, 2 \Re \,\, \langle x,Tx 
\rangle)\\
&=& \tr A + \tr \,\, T^* TA \,\, - 2 \Re \,\, \tr \,\, TA\\
&=& \tr \,\, A \,\, + \tr \,\, TAT^* - 2 \Re \,\, \tr \,\, A^{1/2} TA^{1/2}.
\end{eqnarray*}
If we choose $T = A^{-1} \#B,$ then from \eqref{eq8} we see that $\tr \,\, 
A^{1/2} \,\, TA^{1/2} = \tr \,\, (A^{1/2} BA^{1/2})^{1/2},$ and from 
\eqref{eq9} that $\tr \,\, TAT = \tr \,\, B.$ Thus, for this choice of $T,$ we 
have
\begin{eqnarray*}
 E \|x - \,\, Tx\|^2 &=& \tr \,\, (A+ B) - 2 \, \tr \,\, (A^{1/2} 
BA^{1/2})^{1/2}\\
&=& d^2 (A,B). 
\end{eqnarray*}
Thus the problem
$$\min \,\,  E \|x-y\|^2$$
where $x,y$ are vectors with mean zero and covariance matrices $A$ and $B,$ 
respectively, has as its solution the pairs $(x,y),$ where $x$ is any vector 
and $y=Tx,$ with $T=A^{-1} \# B.$ The matrix $T$ is called the {\it optimal 
transport plan}, or the {\it optimal transport map,} from $A$ to $B.$\\
Let $x$ be a vector with covariance matrix $A,$ and let $y = Tx.$ Then
\begin{eqnarray*}
 E (\overline{y}_i y_j) &=& E \,\, \sum_{k,l} \,\, t_{ik} \,\, t_{kl} 
\,\, \overline{x}_k\,\, x_l \\
&=& \sum_{k,l} \,\, t_{ik} \,\, t_{kl} \,\, a_{kl} \,\, = \,\, (TAT)_{ij}. 
\end{eqnarray*}
If $T$ is the optimal transport map from $A$ to $B,$ then $TAT=B.$ This shows 
that the covariance matrix of the vector $y$ is $B.$\\

The results in this section were proved by Olkin and Pukelsheim \cite{op} and by
Dowson and Landau \cite{dl}. The authoritative reference for optimal transport theory is \cite{v}.
An interesting article explaining connections between optimal transport and Riemannian geometry is \cite{ba}.

\section{Riemannian geometry}

The Bures-Wasserstein distance corresponds to a Riemannian metric, and that is
explained now.\\

From now on we consider positive definite (i.e., nonsingular psd) matrices.
We continue to use the notation $\mathbb {P}(n)$ for the set of all such
matrices. This is an open subset of the real vector space $\mathbb{H}(n).$ Let
$GL(n)$ be the set of all nonsingular matrices. This is an open subset of
$\mathbb{M}(n).$ Both $GL(n)$ and $\mathbb{P}(n)$ are viewed here as
differentiable manifolds. \\

Let $ \pi : GL(n) \rightarrow \mathbb{P}(n) $ be the map defined as $\pi (M) =
MM^*. $ This is a differentiable map, and its derivative $D \pi (M)$ at any
point $M$ is a linear map from $\mathbb{M}(n)$ to $\mathbb{H}(n).$ The action of
this map is
\begin{equation}
 D \pi (M)(Z) = ZM^* + MZ^*, \,\,\,\,Z \in \mathbb{M}(n). \label{eq24}
\end{equation}
The kernel of this map is \\
\begin{eqnarray}
\ker D \pi(M) &=& \left\{ Z : ZM^* + MZ^* = 0\right\}  \nonumber \\
 &=& \left\{ Z : ZM^* \,\,\,\mbox{is skew-Hermitian}\right\}\nonumber  \\
 &=& \left\{ Z = KM^{*-1} : K \,\,\, \mbox{skew-Hermitian} \right\}.
\label{eq25}
\end{eqnarray}
The orthogonal complement of this space with respect to the Frobenius inner
product can be readily computed. A matrix $X$ is in this orthogonal
complement, if and only if we have for all skew-Hermitian matrices $K$
$$ 0 = \langle X, KM^{*-1} \rangle = \re\tr X^*KM^{*-1} = \re\tr M^{*-1} X^* K. $$
This happens if and only if $M^{*-1} X^*$ is Hermitian; i.e. $XM^{-1}$ is
Hermitian. Thus
\begin{equation}
 (\ker D \pi(M))^\perp = \left\{X = HM :H\in \mathbb{H}(n) \right\}.\label{eq26}
\end{equation}
So, we have a direct sum decomposition of the tangent space $T_M GL(n) =
\mathbb{M}(n) $ as
\begin{eqnarray}
 T_M \,\, GL(n) &=& \ker D \pi (M) \oplus (\ker D \pi (M))^\perp \nonumber \\
&=& \mathcal{V}_M \oplus \mathcal{H}_M. \label{eq27}
\end{eqnarray}
The spaces $\mathcal{V}_M$ and $\mathcal{H}_M,$ given by \eqref{eq25} and
\eqref{eq26} are, respectively, called the $\it vertical \,\,space$ and the
$\it horizontal\,\, space $ at $M$ (for the map $\pi$).\\

At this stage we recall two theorems from Riemannian geometry. Let
$(\mathcal{M},g)$ and $(\mathcal{N},h)$ be Riemannian manifolds with
Riemannian metrics $g$ and $h.$ A
differentiable map $\pi : \mathcal{M} \rightarrow \mathcal{N} $ is said to
be a {\it smooth submersion } if its differential $D \pi(m) : T_m
\mathcal{M} \rightarrow T_{\pi(m)} \mathcal{N}$ is surjective at every
point $m.$ Let  $T_m \mathcal{M} = \mathcal{V}_m \oplus
\mathcal{H}_m $ be a decomposition of $T_m \mathcal{M}$ into vertical and
horizontal spaces. Then $\pi$ is called a $\it Riemannian \,\,\,submersion$
if it is a smooth submersion and the map $D \pi(m): \mathcal{H}_m \rightarrow
T_{\pi(m)} \mathcal{N}$ is isometric for all $m.$ \\

\begin{thm}\label{thm3}
Let $(\mathcal{M}, g)$ be a Riemannian manifold.
Let $G$ be a compact Lie group of isometries of $(\mathcal{M}, g)$ acting
freely on $\mathcal{M}.$ Let $\mathcal{N} = \mathcal{M}/G $ and let $\pi :
\mathcal{M} \rightarrow \mathcal{N}$ be the quotient map. Then there exists a
unique Riemannian metric $h$ on $\mathcal{N}$ for which $\pi :
(\mathcal{M},g) \rightarrow (\mathcal{N}, h)$ is a Riemannian submersion.
\end{thm}

\begin{thm}\label{thm4}
Let $(\mathcal{M},g)$ and $(\mathcal{N}, h)$ be
Riemannian manifolds and $\pi : (\mathcal{M},g) \rightarrow (\mathcal{N},
h)$ a Riemannian submersion. Let $\gamma$ be a geodesic in $(\mathcal{M},g)$ such that $\gamma^\prime(0)$ is horizontal. Then
\begin{itemize}
\item[(i)] $ \gamma^\prime (t) $ is
horizontal for all $t.$
\item[(ii)] $\pi \circ \gamma$ is a geodesic in $(\mathcal{N},h)$ of the same length
as $\gamma.$
\end{itemize}
\end{thm} 
See \cite{ghl}.\\

Let us return to our setup now. $GL(n)$ is a Riemannian manifold with the
metric
induced by the Frobenius inner product. The group  $U(n)$ is a compact Lie
group of isometries for this metric. The quotient space $GL(n)/U(n)$ is
$\mathbb{P}(n).$ The metric inherited by the quotient space $\mathbb{P}(n)$
is (upto a constant factor) exactly the one given in Theorem \ref{thm1}; i.e.,
\begin{equation}
\min \|A^{1/2} -B^{1/2} U\|_2 = d(A,B). \label{eq28}
\end{equation}
The map $\pi (M) = MM^*$ is a smooth submersion, as is evident from
\eqref{eq26}. By Theorem \ref{thm3} there is a unique Riemannian metric on
$\mathbb{P}(n)$(for each point $A$ of $\mathbb{P}(n)$ an inner
product $ \langle \cdot , \cdot\rangle_A $ on the tangent space $T_A
\mathbb{P}(n) = \mathbb{H}(n)$) for which $\pi$ is a Riemannian submersion.
To find this inner product we proceed as follows. Let $A=MM^*.$ We
want the map $ D \pi(M) : \mathcal{H}_M \rightarrow T_A \mathbb{P}(n)
= \mathbb{H}(n)$ to be an isometry. The inner product between two
elements $HM$ and $KM$ in the horizonal space $\mathcal{H}_M$ is $\langle HM,KM
\rangle = \re\tr\, KMM^*H=\re\tr\, KAH.$ By \eqref{eq24} we have $D\pi (M) (HM) =
HMM^* + MM^*H = HA + AH.$ So for $D\pi (M)$ to be an isometry the inner product
$ \langle \cdot , \cdot\rangle_A $ on $T_A \mathbb{P}(n) = \mathbb{H}(n)$ must
be given by
\begin{equation}
\langle HA + AH, KA + AK \rangle_A = \re\tr\,KAH. \label{eq29}
\end{equation}
Let $Y$ be any element of $\mathbb{H}(n).$ Then there exists a unique $ H \in
\mathbb{H}(n)$ such that \\
\begin{equation}
HA + AH = Y. \label{eq30}
\end{equation}
Indeed, in an orthonormal basis in which $A = \diag (\alpha_1,
\ldots ,\alpha_n)$ the equation \eqref{eq30} is satisfied by the matrix $H$
with entries
\begin{equation}
h_{ij} = \frac{y_{ij}}{\alpha_i + \alpha_j}. \label{eq31}
\end{equation}
Let $Z$ be another element of $\mathbb{H}(n).$ Then the matrix $k_{ij}=z_{ij}
/(\alpha_i + \alpha_j)$  satisfies the
equation $KA+AK=Z.$ So from \eqref{eq29} and \eqref{eq30} we get\\
\begin{equation}
 \langle Y,Z \rangle_A = \sum_{i,j} \alpha_i \frac{\re\overline{y}_{ji}
z_{ji}}{(\alpha_i + \alpha_j)^2}. \label{eq32}
\end{equation}
To sum up, we have proved the following.\\

\begin{thm}\label{thm5}
For each $A \in \mathbb{P}(n)$  let $\langle Y,Z
\rangle_A$ be the inner product on $\mathbb{H}(n)$ given by \eqref{eq32}. This
gives a Riemannian metric on the manifold $\mathbb{P}(n),$ the distance
function corresponding to which coincides with \eqref{eq28}.
\end{thm}
\vskip0.2in
\begin{center}
\begin{figure}[ht]
\includegraphics{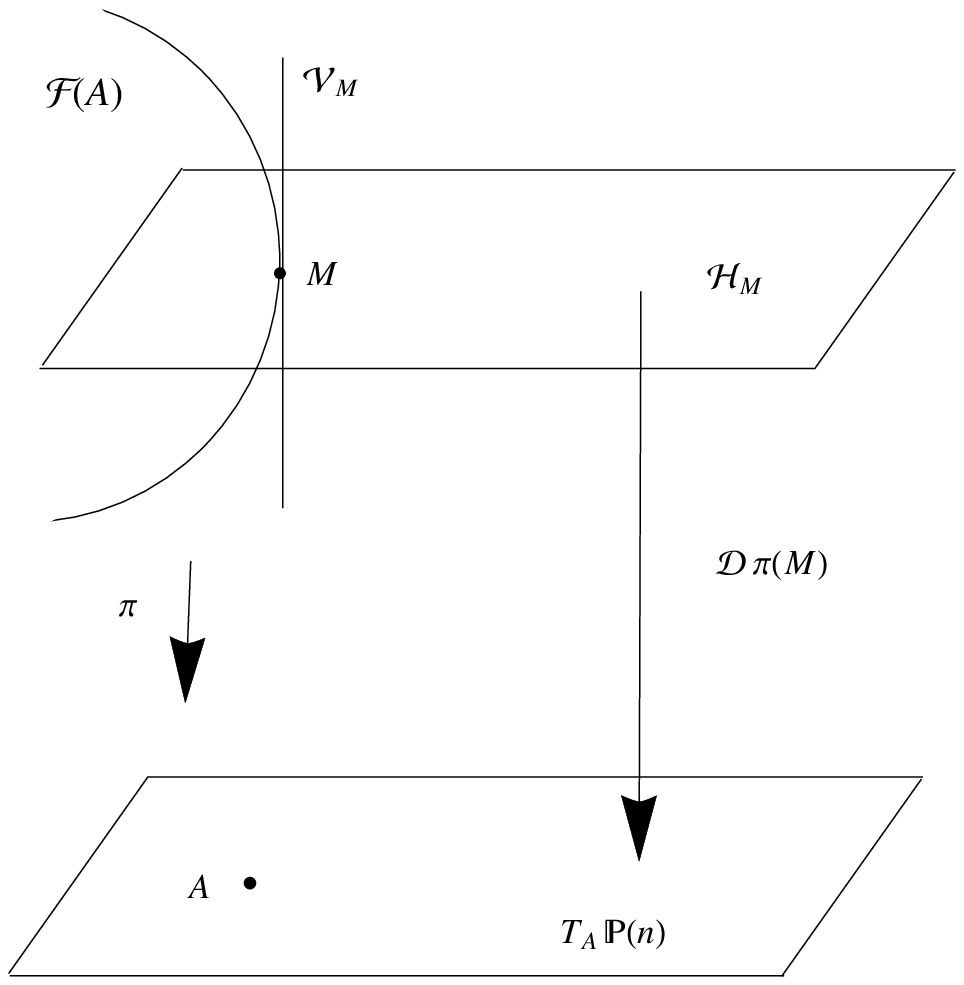}
\caption{}
\end{figure}
\end{center}

Figure 1 is a schematic representation of the Riemannian submersion in Theorem \ref{thm5}.
\vskip.2in

Next, we obtain a formula for the geodesic joining $A$ and $B$ in $\mathbb{P}(n).$ Let $U$ be the unitary polar factor of $B^{1/2}A^{1/2}$; i.e.,
\begin{equation}
B^{1/2}A^{1/2}=U|B^{1/2}A^{1/2}|=U(A^{1/2}BA^{1/2})^{1/2},\label{eq1p}
\end{equation}
and let
\begin{equation}
Z(t)=(1-t)A^{1/2}+tB^{1/2}U, \ \ \ 0\le t\le 1.\label{eq2p}
\end{equation}
From \eqref{eq1p} we have
\begin{eqnarray}
U & = & B^{1/2}A^{1/2}(A^{1/2}BA^{1/2})^{-1/2}\nonumber\\
 & = & B^{1/2}A^{1/2}(A^{1/2}BA^{1/2})^{-1/2}A^{-1/2}A^{1/2}\nonumber\\
 & = & B^{1/2}(AB)^{-1/2}A^{1/2}\nonumber\\
 & = & B^{-1/2}B(B^{-1}A^{-1})^{1/2}A^{1/2}\nonumber\\
 & = & B^{-1/2}(B\#A^{-1})A^{1/2}.\label{eq3p}
\end{eqnarray}
So, the equation \eqref{eq2p} can be written as
\begin{equation}
Z(t)=(1-t)A^{1/2}+t(A^{-1}\# B)A^{1/2},\ \ \ 0\le t\le 1.\label{eq4p}
\end{equation}
We have
\begin{equation}
Z(0)=A^{1/2},\ Z(1)=(A^{-1}\# B)A^{1/2}=B^{1/2}U,\label{eq5p}
\end{equation}
and
\begin{equation}
Z^\prime (t)=B^{1/2}U-A^{1/2}=(A^{-1}\# B-I)A^{1/2},\ \ 0\le t\le 1.\label{eq6p}
\end{equation}
Note that
\begin{equation*}
Z(t)=((1-t)I+t(A^{-1}\# B))A^{1/2},
\end{equation*}
being a product of two positive definite matrices is in $GL(n).$ Being a straight line segment, it is a geodesic. From
\eqref{eq26} and \eqref{eq6p} we see that $Z^\prime (0)$ is in the horizontal space $H_{A^{1/2}}.$ So, by Theorem \ref{thm4}
$\gamma(t)=\pi(Z(t))$ is a geodesic in the space $\mathbb{P}(n)$ with respect to the Riemannian metric \eqref{eq32}.
From
\eqref{eq5p} we see that
\begin{equation*}
\gamma(0)=\pi(Z(0))=A,\ \textrm{ and }
\end{equation*}
\begin{equation*}
\gamma(1)=\pi(Z(1))=Z(1)Z(1)^*=B.
\end{equation*}
Thus $\gamma(t)$ is a geodesic joining $A$ and $B.$
An explicit expression for $\gamma(t)$ can be obtained by using \eqref{eq2p} and \eqref{eq3p}. We have
\begin{eqnarray}
\gamma(t) & = & Z(t)Z(t)^*\nonumber\\
 & = & (1-t)^2A+t^2B+t(1-t)\bigl[A^{1/2}U^*B^{1/2}+B^{1/2}UA^{1/2}\bigr]\nonumber\\
 & = & (1-t)^2A+t^2B+t(1-t)\bigl[A(A^{-1}\# B)+(A^{-1}\# B)A\bigr]\nonumber\\
 & = & (1-t)^2A+t^2B+t(1-t)\bigl[(AB)^{1/2}+(BA)^{1/2}\bigr].\label{eq7p}
\end{eqnarray}
Theorem \ref{thm4} tells us that the length $L_\gamma$ of the geodesic in $\mathbb{P}(n)$ is equal to the
length $L_Z$ in $GL(n).$ The latter is the length of the straight line segment joining $A^{1/2}$
 and $B^{1/2}U.$ So, from Theorem \ref{thm1} we have
\begin{equation*}
L_\gamma=\|A^{1/2}-B^{1/2}U\|_2 = \, d(A,B).
\end{equation*}

We started with the distance $d(A,B)$ on $\mathbb{P}(n)$ and used Theorems \ref{thm3}
and \ref{thm4} to show that this distance corresponds to a Riemannian metric given by
\eqref{eq32}. If, to begin with, we are provided with the metric \eqref{eq32}
at each point $A,$ then starting from it we can obtain the distance function
$d(A,B).$ \\

At the beginning of this section we introduced the vertical and horizontal
spaces at a point $M$ of $GL (n).$ A curve $\widetilde{\gamma}$ in $GL
(n)$ is called ${\it horizontal}$ if for each $t$ the tangent vector
$\widetilde{\gamma}^{\prime} (t)$ is in the horizontal space
$\mathcal{H}_{\widetilde\gamma (t)}.$ From \eqref{eq26}
we see that $\widetilde{\gamma}$ is horizontal if and only if
there
exists a Hermitian matrix $H (t)$ such that \\
\begin{equation}
\widetilde{\gamma}^{\prime}(t) = H(t) \widetilde {\gamma} (t),\,\,\,\,\ 0
\leq t \leq 1. \label{eq37}
\end{equation}
Let
\begin{equation}
\gamma(t) = \widetilde {\gamma}(t)\widetilde {\gamma}
(t)^*.\label{eq38}
\end{equation}
Then $\gamma$ is a curve in $\mathbb{P}(n).$ Differentiating the relation
\eqref{eq38} and then using \eqref{eq37} we see that
\begin{equation}
\gamma^{\prime}(t) = \gamma(t) H(t) + H(t)\gamma(t). \label{eq39}
\end{equation}

If $\gamma$ is any curve in $\mathbb{P}(n),$ then a curve $\widetilde{\gamma}$
in $GL (n)$ is said to be a ${\it horizontal\,\, lift}$ of $\gamma$ if
$\widetilde{\gamma}$ is horizontal and the relation \eqref{eq38} is satisfied.
Every curve $\gamma$ in $\mathbb{P}(n)$ has a unique horizontal lift
$\widetilde{\gamma}$ that satisfies the condition $\widetilde{\gamma}(0)
\,\,\,\widetilde{\gamma}(0)^* = \gamma (0) .$ This can be seen as follows. Given
$\gamma(t)$ let $H(t)$ be the unique solution of the Sylvester equation
\eqref{eq39}. From the smoothness of  $\gamma$ it follows that $H(t)$  is
continous. Let $M$ be a point of $GL (n)$ such that $MM^* = \gamma (0).$ The
initial value problem $ X^{\prime}(t) = H(t) X(t), \,\,\,X(0) = M,$  has a
unique solution. Call this $\widetilde{\gamma}(t).$ We have seen above that
this curve is a horizontal lift of $\gamma (t).$

The length of the curve $\gamma$ is defined as \\
$$ L_{\gamma} = \intop^1_0 \langle \,\,\,\gamma^{\prime} (t),
\,\,\,\gamma^{\prime} (t) \,\,\, \rangle ^{1/2}_{\gamma(t)} \,\,dt. $$
If the inner product in the integrand is defined by \eqref{eq29} and $H(t)$ by
\eqref{eq39}, then this gives\\
$$ L_{\gamma} = \intop^1_0 (\tr \,\, H(t)\,\, \gamma(t) \,\, H(t))^{1/2}\,\,dt.
$$
Using \eqref{eq37} and \eqref{eq38} we obtain from this \\
$$ L_{\gamma} = \intop^1_0 \langle \,\,\,\widetilde{\gamma}^{\prime} (t),
\,\,\,\widetilde{\gamma}^{\prime} (t) \,\,\, \rangle ^{1/2} \,\,dt =
\intop^1_0 \|\widetilde{\gamma}^{\prime}(t)\|_2 \,\, dt. $$
This is the length of the curve $\widetilde{\gamma}$ with respect to the
Euclidean distance, and cannot be smaller than the straight line distance. So
$$ L_{\gamma} \geq\,\, \|\widetilde{\gamma} (0)\,\, -
\,\, \widetilde{\gamma} (1)\,\, \|_2.$$
If  $\gamma (0) = A$ and $\gamma (1) = B,$ then $\widetilde{\gamma}(0)$
and $\widetilde{\gamma}(1)$ are points in $\mathcal{F}(A)$ and
$\mathcal{F}(B),$ respectively. So, by Theorem \ref{thm1}
$$ L_{\gamma} \geq\,\, d(A,B).$$
Earlier we have seen a curve for which the two sides of this inequality are
equal. Thus the metric \eqref{eq32} leads to the distance function $d(A,B)$ by a 
direct computation.

The material in this section is based on \cite{bz, j, t, u}.
Takatsu \cite{t} also discusses the metric geometry of the spaces of psd matrices of rank $k,$
$1\le k\le n.$
A very interesting research paper by K. Modin \cite{mo} dicusses the connections
between optimal transport, geometry and matrix decompositions.

\section{The Wasserstein mean}

There is another standard metric on $\mathbb{P}(n)$ which has been extensively studied. In this the inner product on the tangent space $T_A\mathbb{P}(n)=\mathbb{H}(n)$ is given by
\begin{equation}
\langle Y, Z\rangle_A=\tr\, A^{-1}YA^{-1}Z,\label{eq40}
\end{equation}
and the associated distance function is
\begin{equation}
\delta(A,B)=\|\log\, A^{-1/2}BA^{-1/2}\|_2.\label{eq41}
\end{equation}
Any two points $A, B$ of $\mathbb{P}(n)$ can be joined by a unique geodesic with respect to this metric, and a natural parametrisation for this geodesic is
\begin{equation}
A\#_tB=A^{1/2}(A^{-1/2}BA^{-1/2})^tA^{1/2},\ \ \ \ 0\le t\le 1.\label{eq42}
\end{equation}
The geometric mean $A\# B$ defined in \eqref{eq8} is evidently the midpoint of this geodesic; i.e.,
\begin{equation*}
A\# B=A\#_{1/2}B.
\end{equation*}
The metric \eqref{eq41} has lots of isometries. We have
\begin{equation}
\delta(XAX^*,XBX^*)=\delta(A,B)\textrm{ for all }X\in GL(n),\label{eq43}
\end{equation}
and
\begin{equation}
\delta(A^{-1},B^{-1})=\delta(A,B)\textrm{ for all }A,B.\label{eq44}
\end{equation}
This bestows upon the geometric mean $A\# B$ several interesting and useful properties, and the object is much used in operator theory, quantum mechanics, electrical networks, elasticity, image processing, etc.
The collection \cite{nb} has several articles on the theory,
computation, and applications of this mean and its multivariable version.

It is natural to ask what properties the $``$mean'' with respect to the distance \eqref{eq1} might have. Let us adopt the notation $A\diamond_tB$ for the
geodesic $\gamma(t)$ given in \eqref{eq7p}. The midpoint of this is
\begin{equation}
A\diamond B=\frac{1}{4}(A+B+(AB)^{1/2}+(BA)^{1/2}).\label{eq45}
\end{equation}
We call this the {\it Wasserstein mean} of $A$ and $B.$ The relations
\begin{equation}
(AB)^{1/2}=A(A^{-1}\# B)=A^{1/2}(A^{1/2}BA^{1/2})^{1/2}A^{-1/2},\label{eq46}
\end{equation}
will be used in the following discussion.

For the Bures-Wasserstein distance \eqref{eq1} only a very restrictive version of \eqref{eq43} is true: we have $d(UAU^*,UBU^*)=d(A,B)$ provided $U$ is unitary. The analogue of \eqref{eq44} is not valid for $d.$ So the Wasserstein mean does not have many of the interesting properties of the mean $A\# B.$ The following theorem is, therefore, surprising. Recall the operator version of the harmonic-geometric-arithmetic mean inequality. This says
\begin{equation}
\bigl(\frac{A^{-1}+B^{-1}}{2}\bigr)^{-1}\le A\# B\le \frac{A+B}{2}.\label{eq47}
\end{equation}
The second inequality in \eqref{eq47} can be extended as
\begin{equation}
A\#_tB\le (1-t)A+tB,\ \ \ 0\le t\le 1.\label{eq48}
\end{equation}
This has an analogue for the Wasserstein mean:

\begin{thm}\label{thm6}
For all positive definite matrices $A$ and $B$ we have
\begin{equation}
A\diamond_tB\le (1-t)A+tB,\ \ \ 0\le t\le 1.\label{eq49}
\end{equation}
\end{thm}

\begin{proof}
Using the equations \eqref{eq7p} and \eqref{eq46} we have
\begin{eqnarray*}
A\diamond_t B & = & \gamma(t)\\
& = & (1-t)^2A+t^2B\\
&  & \ \ +t(1-t)\left[  A^{1/2}(A^{1/2}BA^{1/2})^{1/2}A^{-1/2}+A^{-1/2}(A^{1/2}BA^{1/2})^{1/2}A^{1/2}\right]\\
 & = & A^{-1/2}\bigl[(1-t)^2A^2+t^2A^{1/2}BA^{1/2}\\
 &  & \ \ t(1-t)\{A(A^{1/2}BA^{1/2})^{1/2}+(A^{1/2}BA^{1/2})^{1/2}A\}\bigr]A^{-1/2}\\
 & = & A^{-1/2}\bigl[ (1-t)A+t(A^{1/2}BA^{1/2})^{1/2}\bigr]^2A^{-1/2}.
\end{eqnarray*}
The map $f(X)=X^2$ is matrix convex; i.e., for all Hermitian matrices $X$ and $Y$ we have
\begin{equation*}
\left[ (1-t)X+tY\right]^2\le (1-t)X^2+tY^2,\ \ \ 0\le t\le 1.
\end{equation*}
Hence,
\begin{eqnarray*}
A\diamond_tB & \le & A^{-1/2}\left[ (1-t)A^2+tA^{1/2}BA^{1/2}\right] A^{-1/2}\\
 & = & (1-t)A+tB.
\end{eqnarray*}
This proves the inequality \eqref{eq49}.
\end{proof}

Another instructive proof of Theorem \ref{thm6} goes as follows. Using the inequality
\begin{equation*}
0\le A^{-1/2}(A-(A^{1/2}BA^{1/2})^{1/2})^2 A^{-1/2},
\end{equation*}
and \eqref{eq46} we obtain
\begin{equation}
(AB)^{1/2}+(BA)^{1/2}\le A+B.\label{eq50a}
\end{equation}
So, from \eqref{eq45} we have
\begin{equation}
A\diamond B\le \frac{1}{2}(A+B).\label{eq50b}
\end{equation}
Since $A\diamond_t B=\gamma(t)$ is the natural parametrisation of the geodesic joining $A$ and $B,$ we have
\begin{equation*}
(A\diamond_sB)\diamond_u(A\diamond_tB)=A\diamond_v B,
\end{equation*}
where $v=(1-u)s+ut$ for all $s,t,u$ in $[0,1].$ Using this we can obtain from \eqref{eq50b} the inequality \eqref{eq49} for all dyadic rational values of $t.$ By continuity it holds for all $0\le t\le 1.$

Theorem \ref{thm6} may lead us to expect that the inequality
\begin{equation}
\bigl(\frac{A^{-1}+B^{-1}}{2}\bigr)^{-1}\le A\diamond B,\label{eq52b}
\end{equation}
 might also be true. However, this is not always the case.

If $A$ and $B$ are two positive definite matrices such that $A\le B,$ then it follows from Theorem \ref{thm6} that $A\diamond B\le B.$ 
However, it is not necessary that $A\le A\diamond B.$ If we choose
\begin{equation*}
A=\begin{bmatrix}1 & 1\\
1 & 2\end{bmatrix},\ \ B=\begin{bmatrix}3 & 1\\
1 & 2\end{bmatrix},
\end{equation*}
then
\begin{equation*}
A\diamond B\approx\begin{bmatrix}1.8495 & 1.0449\\
1.0449 & 1.9857\end{bmatrix},
\end{equation*}
and $A\nleq A\diamond B.$

This example also shows that $A\diamond B$ is not monotone with respect to the partial order $\le$; i.e. if $A\le A^\prime,$ then it is not necessary that $A\diamond B\le A^\prime\diamond B.$

\section{The Wasserstein barycentre}
Let $A_1,\ldots, A_m$ be elements of $\mathbb{P}(n)$ and let $w =
(w_1,\ldots,w_m)$ be a weight vector ; i.e., $w_j > 0$ and $\Sigma \,\, w_j =
1.$ Consider the minimisation problem
\begin{equation}
 \underset{X>0}{\min}\,\,\, \sum\limits^{m}_{j=1} w_j d^2(X,A_j).
\label{eq53}
\end{equation}

This problem was first considered by Knott and Smith \cite{ks} as a multivariable
generalisation of the work of Olkin and Pukelsheim discussed in Section 3
above. Agueh and Carlier \cite{ac} studied the general problem of determining
the barycentre of several probability measures on $\mathbb{R}^n.$ The special
case of Gaussian measures is the problem \eqref{eq53}. The general problem
has been studied as a part of the the {\it multimarginal transport problem} or the {\it
$m$-coupling problem.} \\

Theorem 6.1 of \cite{ac} says that the problem \eqref{eq53} has a unique
solution. The proof of uniqueness in \cite{ac}, that draws on the earlier 
discussion of
the general case, relies on tools from nonsmooth analysis, convex duality and
the theory of optimal transport. In the spirit of this paper we now provide
another proof using simple ideas from matrix analysis. \\

The minimiser in \eqref{eq53} is called the {\it Wasserstein barycentre} of
$A_1 \ldots, A_m$ with weights $w_1,\ldots,w_m.$ This is the positive definite
matrix
\begin{equation}
\Omega (w; A_1, \ldots, A_m) = \underset{X>0}{\argmin}\,\,\,
\sum\limits^{m}_{j=1} w_j d^2(X,A_j). \label{eq54}
\end{equation}

Using the definition \eqref{eq1} we see that the objective function in
\eqref{eq54} is $f(X)$, where
\begin{equation}
 f(X) = \sum\limits^{m}_{j=1} \,\, w_j\,\, \tr A_j + \sum\limits^{m}_{j=1} \,\,
w_j \,\,\tr (X-2(A^{1/2}_j X
A^{1/2}_j)^{1/2}).\label{eq55}
\end{equation}
This is a differentiable function on the convex cone $\mathbb{P}(n).$ We will 
calculate the derivative of $f,$ and show that there is a point in 
$\mathbb{P}(n)$ at which is vanishes. This local minimum for $f$ will be a 
(unique) global minimum if $f$ is a (strictly) convex function. From 
\eqref{eq55} it is clear that to prove strict convexity of $f$ it
is enough to establish strict concavity of the function $h(X) = \tr
\,\,X^{1/2}.$
This is our next theorem.

\begin{thm}\label{thm7}
The map $h(X) = \tr \,\,X^{1/2} $ from
$\mathbb{P}(n)$ into $(0, \infty)$ is strictly concave; i.e., if
$X$ and $Y$ are two distinct elemens of $\mathbb{P}(n)$ and $\alpha, \beta $
are positive numbers with $\alpha + \beta = 1,$ then
\begin{equation}
 h(\alpha X +\beta Y) > \alpha h(X) +\beta h (Y)\label{eq56}.
\end{equation}
\end{thm}

\begin{proof}
It is well-known that $X \longmapsto X^{1/2} $ is an operator concave
function. See
Chapter V of \cite{rbh}. So, we have
$$(\alpha X + \beta Y)^{1/2} \geq \alpha X^{1/2} + \beta Y^{1/2},$$
and hence
$$ \tr (\alpha X + \beta Y)^{1/2} \geq \alpha \,\, \tr X^{1/2} + \beta \,\, \tr
\,\, Y^{1/2}.$$
We have to show that in this last inequality the two sides cannot be equal if
$X \neq Y.$ Suppose
$$ \tr \left[(\alpha X + \beta Y)^{1/2} - (\alpha X^{1/2} + \beta Y^{1/2})
\right] = 0.$$
The matrix inside the square brackets is positive semidefinite. So, its trace
can be zero only if
$$(\alpha X + \beta Y)^{1/2} = \alpha X^{1/2} + \beta Y^{1/2}.$$
Square both sides, and then use the relations $ \alpha - \alpha^2 = \beta -
\beta^2 = \alpha \beta,$ to obtain
$$ \alpha \beta (X + Y - X^{1/2} Y^{1/2} - Y^{1/2} X^{1/2}) = 0. $$
Since $\alpha \beta \neq 0,$, this gives\\
$$ (X^{1/2} - Y^{1/2})^2 = 0,$$
and hence $X^{1/2} = Y^{1/2},$ and $X=Y.$
\end{proof}
Now we show that $f$ does have a minimum in $\mathbb{P}(n)$ by evaluating
the derivative $Df(X)$ and equating that to zero. A convenient summary of facts
about matrix differential calculus can be found in Chapter X of \cite{rbh}. \\

The nonlinear term in \eqref{eq55} is $ g(X) = (A^{1/2} X A^{1/2})^{1/2}.$ We
evaluate $ Dg(X)$ from first principles. The derivative of the function $\Psi
(A) =A^2$ is the linear map $D \Psi (A)$ defined as $D \Psi(A) (X) = AX +XA.$
The function $\varphi (A) = A^{1/2}$ on $\mathbb{P}(n)$ is the inverse of
$\Psi.$ Hence $ D \varphi (A) = \left[ D\Psi (\varphi (A))\right]^{-1} = \left[
D\Psi (A^{1/2})\right]^{-1}. $ Thus $D \varphi (A)$ is the inverse of the linear
map $X \longmapsto A^{1/2}X + X A^{1/2}. $ By well known facts about the
Sylvester matrix equation (see \cite{rbh} or \cite{br}) this inverse is given by the
formula
$$ D \varphi (A)(X) = \intop_0^\infty \,\,\, e ^{-tA^{1/2}} X \,\,e
^{-tA^{1/2}} \,\, dt.$$
Let $\lambda (X) = A^{1/2} X A^{1/2} .$ Then $g$ is the composite
$\varphi \circ \lambda.$ So, by the chain rule of differentiation,\\
\begin{eqnarray*}
D g(X)(Y) &=& (D\varphi (\lambda (X)) \circ D \lambda (X))(Y)\\
&=& D\varphi (A^{1/2} X A^{1/2})  (A^{1/2} \,\,Y \,\, A^{1/2})\\
&=&\intop_0^\infty \,\,\, e ^{-t(A^{1/2} X \,\,A^{1/2})^{1/2}}
(A^{1/2} YA^{1/2}) \,\,e^{-t(A^{1/2} X A^{1/2})^{1/2}}\,\, dt.
\end{eqnarray*}
Taking traces, and using the cyclicity of trace, we get
\begin{eqnarray*}
\tr \,\, D g (X)(Y) &=&\intop_0^\infty (\tr \,\, A^{1/2}
e^{-2t(A^{1/2} X \,\, A ^{1/2})^{1/2}} A^{1/2} Y)\,\, dt \\
&=& \tr \,\, A ^{1/2} (\intop_0^\infty e^{-2t(A^{1/2} X \,\,A^{1/2})^{1/2}}
\,\, dt) A^{1/2} Y.\\
\end{eqnarray*}
The last integral above is equal to $\frac{1}{2} (A^{1/2} X
A^{1/2})^{-1/2}.$ (Use the fact that $\intop_0^\infty e^{-t\alpha} \,\,dt =
\frac{1}{\alpha}$ for every $\alpha > 0.$) Hence
\begin{eqnarray*}
\tr D g(X)(Y) &=& \frac{1}{2} \tr A^{1/2} \left( A^{-1/2} X^{-1}
A^{-1/2}\right)^{1/2} A^{1/2}\,\, Y \\
&=& \frac{1}{2} (A \# X^{-1})Y.\\
\end{eqnarray*}
So, from \eqref{eq55} we see that
\begin{eqnarray*}
 Df(X)(Y) &=& \sum\limits^{m}_{j=1} w_j\tr (Y-(A_j \#
X^{-1})Y)\\
&=& \tr \left( I-\sum\limits^{m}_{j=1} w_j \left( A_j \#
X^{-1}\right) \right)Y.
\end{eqnarray*}
Thus $Df (X) = 0$ if and only if
\begin{equation}
I = \sum\limits^{m}_{j=1} w_j (A_j \# X^{-1}). \label{eq57}
\end{equation}
This is equivalent to saying
\begin{equation}
 X = \sum\limits^{m}_{j=1} w_j (X^{1/2} A_j X^{1/2})^{1/2}.
\label{eq58}
\end{equation}
Finally, we show that there exists a point $X$ in $\mathbb{P}(n)$ that 
satisfies the equation \eqref{eq58}. Indeed, if $\alpha I \leq A_j \leq \beta 
I,$ for all $ 1 \leq j \leq m,$ then this $X$ belongs to the compact convex set 
$\mathcal{K} = \{ X \in \mathbb{P}(n) :\alpha I \leq X \leq \beta I \}.$ To 
see this consider the function
$$F(X) = \sum\limits^{m}_{j=1} w_j (X^{1/2} A_j X^{1/2})^{1/2}.$$
Then note that $(X^{1/2}A_jX^{1/2})^{1/2}\le (\beta X)^{1/2}\le \beta I$ for all $X\in\mathcal{K}.$
By the same reasoning $(X^{1/2}A_jX^{1/2})^{1/2}\ge \alpha I$ for all $X\in\mathcal{K}.$ This shows that 
$F$ maps $\mathcal{K}$ into itself. By Brouwer's fixed point theorem there 
exists a point $X$ in  $\mathcal{K}$ such that $F(X)=X.$ This $X$ is a solution 
of the equation \eqref{eq58}.\\

We have proved the following theorem first obtained in \cite{ac}, building upon 
the earlier work in \cite{ks} and \cite{r}

\begin{thm}\label{thm8}
The minimisation problem \eqref{eq54} has a unique
solution
which is also the solution of the nonlinear matrix equation \eqref{eq58}.
\end{thm}

We do not know how to obtain the solution of \eqref{eq58} in an explicit form.
In the special case $m=2,$ with $A_1 = A,\,\, A_2 =B, $ and $(w_1, w_2) =
(1-t,t),$ the equation \eqref{eq57} reduces to
\begin{equation}
 I = (1-t)(A \# X^{-1}) + t (B \# X^{-1}).
\label{eq62}
\end{equation}
The solution to this equation is
\begin{equation*}
X=\argmin\left[(1-t)d^2(X,A)+td^2(X,B)\right].
\end{equation*}
By the definition of geodesics with respect to the metric $d,$
such an $X$ is the unique point on the geodesic segment joining $A$ and $B$ at distance
$td(A,B)$ from $A.$
In other words 
\begin{equation}
 X = A \diamond_t B = (1-t)^2 A + t^2 B + t (1-t) \left[(AB)^{1/2} +
(BA)^{1/2}\right]. \label{eq60}
\end{equation}

The equation \eqref{eq58} can be used to obtain some important order properties
of the Wasserstein barycentre. The next theorem is a multivariable analogue of
Theorem \ref{thm6}.

\begin{thm}\label{thm9}
Let $A_1, \ldots, A_m$ be positive definite matrices and
let $w= (w_1,\ldots, w_m)$ be a weight vector. Then
\begin{equation}
 \Omega (w; A_1,\ldots, A_m) \leq \sum\limits^{m}_{j=1} w_j A_j. \label{eq61}
\end{equation}
\end{thm}

\begin{proof}
 The matrix $ \Omega = \Omega (w; A_1, \ldots, A_m)$ obeys the relation
$$\Omega = \sum\limits^{m}_{j=1} w_j (\Omega^{1/2} A_j \Omega^{1/2})^{1/2}.$$
Square both sides and then use the fact that the function$f(A) = A^2 $ is
matrix convex. This gives
$$\Omega^2 \leq \sum\limits^{m}_{j=1} w_j \Omega^{1/2} A_j
\Omega^{1/2} = \Omega^{1/2} \left( \sum\limits^{m}_{j=1} w_j A_j \right)
\Omega^{1/2}.$$
The inequality \eqref{eq61} follows from this.
\end{proof}
Theorem 9 is much stronger than the known inequality $\tr \,\, \Omega \,\, \leq 
\,\, \tr \,\, \sum w_j A_j,$ which has been proved in \cite{am}. (See the last 
inequality in Theorem 4.2 there.)\\
\section{The $m$-Coupling Problem}
We explain briefly how the Wasserstein barycentre is useful in solving the 
several variable version of the problem considered in Section 3.\\

Let $x_1, \ldots, x_m$ be random vectors in $\mathbb{C}^n,$ each having zero 
mean, and with covariance matrices $A_1, \ldots, A_m.$ We are asked to find a 
tuple $( x_1, \ldots, x_m)$ that solves the minimisation problem
\begin{equation}
\min \,\, E \sum_{1 < j} \| x_i - x_j\|^2. \label{eq65}
\end{equation}
This is the same problem as the one of maximising $E\|\sum x_j\|^2.$ A little 
more generally, we consider the problem
\begin{equation}
\max \,\, E \| \sum_{j=1}^m w_j \,\, x_j\|^2, \label{eq66}
\end{equation}
where $w_1, \ldots, w_m$ are given weights.\\

Let $\Omega = \Omega (w; A_1, \ldots, A_m)$ and let $R_j = \Omega^{-1} \# A_j 
\,\, , \,\, 1 \leq j \leq m.$ Let $z = \sum w_jx_j.$ Then $\langle z,z \rangle 
\,\,= \,\, \sum w_j \langle z,x_j \rangle.$ If $T$ is any positive definite 
matrix, and $x,y$ any two vectors, then using the Schwarz inequality and the 
arithmetic-geometric mean inequality, we see that
\begin{eqnarray}
\langle x, y \rangle &=& \langle T^{1/2} x, \, T^{-1/2} y \rangle \,\, 
\leq \|T^{1/2} x\| \, \| T^{-1/2} y\| \nonumber\\
&=&  \langle x, T x \rangle^{1/2}   \,\, \langle y, T^{-1} y\rangle^{1/2}  
\nonumber\\
&\leq& \frac{1}{2} \left[ \langle x, Tx \rangle \,\, + \,\, \langle y, T^{-1} y 
\rangle \right]. \label{eq67} 
\end{eqnarray}
Hence,
\begin{eqnarray*}
 \langle z, z \rangle \,\, \leq \frac{1}{2} \left[ \sum_{j=1}^m \,\, w_j \,\, 
\langle z, R_j z \rangle + \sum_{j=1}^m \,\, w_j \,\, \langle x_j, R_j^{-1} x_j 
\rangle \right].
\end{eqnarray*}
From \eqref{eq57} we know that $\sum_{j=1}^m \,\, w_j \,R_j = I.$ So, the 
inequality above yields
\begin{eqnarray*}
 \langle z, z \rangle \,\, \leq \sum_{j=1}^m \,\, w_j \,\, 
\langle x_j, R_j^{-1} x_j \rangle.
\end{eqnarray*}
Thus
\begin{eqnarray*}
E \|z\|^2 \,\, \leq \,\, \sum_{j=1}^m \,\, w_j \,\, E
\langle x_j, R_j^{-1} x_j \rangle.
\end{eqnarray*}
Since $x_j$ has covariance matrix $A_j,$ this gives
\begin{eqnarray}
 E \|z\|^2 \,\, &\leq& \,\, \sum_{j=1}^m \,\, w_j \,\, \tr \,\,R_j^{-1} A_j 
\nonumber\\
&=& \sum_{j=1}^m \,\, w_j \,\, \tr \,\,A_j^{1/2} R_j^{-1} A_j^{1/2} 
\nonumber\\
&=& \sum_{j=1}^m \,\, w_j \,\, \tr \,\,A_j^{1/2} (\Omega \# A_j^{-1}) 
A_j^{1/2}\nonumber\\
&=& \sum_{j=1}^m \,\, w_j \,\, \tr \,\,(A_j^{1/2} \Omega A_j^{1/2}) \# I
\nonumber\\
&=& \tr \,\, \sum_{j=1}^m \,\, w_j \,\, (A_j^{1/2} \Omega A_j^{1/2})^{1/2}
\nonumber\\
&=& \tr \,\, \Omega. \label{eq68}
\end{eqnarray}
Note that both the inequalities in \eqref{eq67} are equalties if $y = Tx.$ 
Hence, there is equality at the first step in \eqref{eq68} if $z = R_j x_j$ 
for $1 \leq j \leq m.$ This can be achieved by choosing $x_1$ arbitrarily and 
then putting $x_j = R_jR_1^{-1} x_1$ for $2 \leq j \leq m.$\\

To sum up, we have shown that the problem \eqref{eq66} has the solution
\begin{equation}
 \max E\|  \sum_{j=1}^m  w_j \,\, x_j \|^2 = \tr \,\, \Omega  (w ; A_1, 
\ldots , A_m). \label{eq69} 
\end{equation}
The maximum is attained at every m-tuple
\begin{equation}
(x_1, \,\, R_2 R_1^{-1} \,\, x_1 , \,\, R_3 R_1^{-1} \,\, x_1 , 
\ldots, R_m R_1^{-1}\,\, x_1 ),\label{eq70}
\end{equation}
where $x_1$ is chosen arbitrarily subject to the given conditions that it has 
mean $0$ and covariance matrix $A_1.$ Note that, then we have
\begin{eqnarray} 
\sum_{j=1}^m \,\, w_j x_j &=& w_1 x_1 \,\, + \,\, \sum_{j=2}^m \,\, w_j 
 R_j  R_1^{-1}\,\, x_1 \nonumber \\
&=& \sum_{j=1}^m \,\, w_j  R_j  R_1^{-1}\,\, x_1 \,\, = \,\, 
R_1^{-1}\,\, x_1, \label{eq71} 
\end{eqnarray}
the last equality being a consequence of the fact that $\sum \limits_{j=1}^m 
\,\, w_j R_j = I.$ The maps $R_jR_1^{-1}$ are said to provide an optimal 
coupling between $x_1, \ldots, x_m$ that occur as a solution of \eqref{eq66}.\\

Many of the ideas presented in Sections 6 and 7 go back to the paper of Knott 
and Smith \cite{ks}. Among other things, the matrix equation \eqref{eq58}, that 
a solution to the minimisation problem \eqref{eq54} must satisfy, is derived 
there. However, questions about the existence and uniqueness of solutions of 
this equation are not settled in this paper. The existence was established by 
Ruschendorf and Uckelmann in \cite{r}, and the uniqueness by Agueh and Carlier 
in \cite{ac}. The elegant argument using Brouwer's fixed point theorem to 
establish the existence of a solution occurs in \cite{ac}, and we have adopted 
it verbatim. Our proof of uniqueness is different, and uses ideas more familiar 
in matrix analysis. We must add that the problem studied in \cite{ac} is the 
more general problem of the barycentre of measures. The matrix case that we are 
discussing corresponds to the special Gaussian measures.

\section{Computing the Barycentre}
Whereas for two matrices $A$ and $B$ their barycentre is given by an explicit 
formula \eqref{eq48}, no such formula is known in the case of three or more 
matrices. We know only that $\Omega$ is the unique solution of the equation 
\eqref{eq57}, or equivalently of \eqref{eq58}. The latter suggests that it may 
be possible to compute $\Omega$ by a fixed point iteration. Such an iteration 
has been developed in a very interesting recent paper \cite{am}. In this 
section we explain the main ideas of this paper, restricting ourselves to matrix 
analytic techniques, and simplifying some proofs.\\

Throughout this section $A_1, \ldots, A_m$ are given positive definite matrices 
and $w= (w_1, \ldots, w_m)$ a given set of weights. For each $A \in 
\mathbb{P}(n)$ let
\begin{eqnarray}
H_j (A) &=& A^{-1} \# A_j\,\, , \,\,\,\,\,\,\,\ 1 \leq j \leq m, \label{eq72}\\
H(A) &=& \sum_{j=1}^m \,\, w_j \,\, H_j (A), \label{eq73}\\
K(A) &=& A^{-1/2} \left(\sum_{j=1}^m \,\, w_j \,\,(A^{1/2} A_j 
A^{1/2})^{1/2}\right)^2 A^{-1/2}. \label{eq74}
\end{eqnarray}
We note that 
\begin{eqnarray}
 K(A) = H (A) AH (A). \label{eq75}
\end{eqnarray}
Also, note that
\begin{eqnarray}
A^{-1} \# K(A) &=& A^{-1/2} \left( A^{1/2} \,\, K(A) A^{1/2} \right)^{1/2} 
A^{-1/2} \nonumber \\
&=& A^{-1/2} \left( \sum_{j=1}^m \,\, w_j \,\,(A^{1/2} A_j 
A^{1/2})^{1/2} \right) A^{-1/2} \nonumber \\
&=& \sum \,\, w_j \,\, H_j (A) \,\, = H(A). \label{eq76}
\end{eqnarray}
Equations \eqref{eq72} and \eqref{eq76} say that $H_j (A)$ and $H(A)$ are the 
optimal transport maps from $A$ to $A_j$ and to $K(A),$ respectively. We define 
the {\it variance} of $A$ as
\begin{eqnarray}
 V(A) = \sum_{j=1}^m \,\, w_j \, d^2 (A, A_j). \label{eq77}
\end{eqnarray}
The following variance inequality is a rephrasing in our context of Proposition 
3.3 in \cite{am}.
\begin{thm} \label{thm10}
 For every positive definite matrix $A$ we have
\end{thm}
\begin{equation}
 V(A) \geq V(K(A)) \,\, + \,\, d^2 (A,K(A)). \label{eq78}
\end{equation}
\begin{proof}
 Let $y_1, \ldots, y_m$ be vectors in $\mathbb{C}^n$ and let $\overline{y} = 
\sum\limits_{j=1}^m \,\, w_j y_j$ be their weighted arithmetic mean. Then for 
every $x 
\in \mathbb{C}^n$ we have
\begin{equation}
\sum_{j=1}^m \,\, w_j \,\, \|x-y_j \|^2 \,\,= \,\, \sum_{j=1}^m \,\, w_j \,\, 
\|\overline{y}-y_j \|^2 + \|x-\overline{y}\|^2. \label{eq79}
\end{equation}
(This is the variance equality in Euclidean space that \eqref{eq78} mimics. The 
Euclidean distance is replaced by the metric $d,$ the points $y_j$ by the 
matrices $A_j,$ the mean $\overline{y}$ by $K(A),$ and we have an inequality in 
place of equality.)\\

Choose a vector $x$ in $\mathbb{C}^n$ with mean $0$ and covariance matrix $A.$ 
For $1\leq j\leq m,$ let $y_j = H_j(A)x,$ we have from the results in Section 3
$$ d^2 (A,A_j) \,\, = \,\, E \,\, \|x-H_j(A)x\|^2 \,\, = \,\, E 
\,\, \|x-y_j\|^2.$$
Hence,
\begin{equation}
 V(A) = \sum_{j=1}^m \,\, w_j \,\, E\|x-y_j\|^2. \label{eq80}
\end{equation}
Similarly, since H(A) is the optimal transport map from $A$ to $K(A),$ we have
$$ d^2 (A,K(A)) \,\, = \,\, E \,\, \|x-H(A)x\|^2.$$
But $H(A)x \,\, = \,\, \sum w_j \, H_j (A)x \,\, = \,\, \sum w_j y_j \,\, 
= \overline{y}.$ So, 
\begin{equation}
d^2 (A,K(A)) \,\, = \,\, E \,\, \|x-\overline{y}\|^2. \label{eq81}
\end{equation}
 Since $H(A)$ is the transport map from $A$ to $K(A)$ and $x$ has covariance 
matrix $A,$ it follows that $\overline{y}$ has $K(A)$ as its covariance matrix. 
Hence
\begin{equation}
 E \,\, \|\overline{y}-y_j\|^2 \,\, \geq d^2 (K(A), A_j)
, \,\,\ 1 \leq j \leq m. \label{eq82}
\end{equation}
The relations \eqref{eq79}-\eqref{eq82} put together lead to the inequality 
\eqref{eq78}.
\end{proof}
{\it Remark.} Using the definition of the variance $V(A)$ and of the metric 
$d(A,B)$ it can be seen that the inequality \eqref{eq78} is equivalent to the 
trace inequality
\begin{equation}
 \sum_{j=1}^m \,\, w_j \,\, \tr \,\, (A^{1/2}_j\,\, K(A) \,\, A_j^{1/2})^{1/2} 
\,\, \geq \,\, \tr \,\, K(A). \label{eq83}
\end{equation}
It might be very difficult to prove this using the usual matrix analysis 
arguments. The very special case $A=I$ of \eqref{eq83} says
\begin{equation}
 \tr \,\, \sum \,\, w_j \left( A_j^{1/2} \left(\sum w_j A_j^{1/2}\right)^2 \, 
A_j^{1/2} \right)^{1/2} \,\, \geq \,\, \tr \left( \sum \,\, w_j A_j^{1/2} 
\right)^2. \label{eq84}
\end{equation}
From the inequality (IX.11) on page 258 of \cite{rbh} we have
\begin{eqnarray*}
\tr \,\, \left( A_j^{1/2} \left( \sum \,\, w_j \, A_j^{1/2} \right)^2 
A_j^{1/2}\right)^{1/2} &\geq& \tr \,\, A_j^{1/4} \left( \sum \,\, w_j 
A_j^{1/2} 
\right)A_j^{1/4}.\\
&=& \tr \,\, A_j^{1/2} \left( \sum \,\, w_j A_j^{1/2} \right).
\end{eqnarray*}
The inequality \eqref{eq84} follows from this. So, even the special case $A=I$ 
of \eqref{eq83} needs rather intricate arguments. Results proved in the context 
of optimal transport could thus add to the tools used in deriving matrix 
inequalities.\\

The next theorem is the main result (Theorem 4.2) of \cite{am}. Some steps in 
the 
proof have been simplified.
\begin{thm} \label{thm11}
 Let $S_o$ be any positive definite matrix and for $n \geq 0$ define $S_{n+1} = 
K(S_n),$ where $K$ is the map defined in \eqref{eq74}. Then 
\end{thm}
\begin{eqnarray}
{\rm (i)}&{}& \qquad \underset{n \rightarrow \infty}{\lim}\,\,\, S_n
\,\, = \Omega
\nonumber\\
{\rm (ii)}&{}& \qquad \tr \, S_n \leq \, \tr \, S_{n + 1} \,\, \leq
\,\, \tr \, \Omega \,\, \mbox{for all}\,\, n \geq 1.\nonumber
\end{eqnarray}
\begin{proof}
 By definition
$$ S_{n+1} \,\, = \,\, S_n^{-1/2} \left( \sum_{j=1}^m \,\, w_j \left(S_n^{1/2} 
A_j \,\, S_n^{1/2}\right)^{1/2}\right)^2 \, S_n^{-1/2}.$$
The square function is matrix convex. Hence,
\begin{eqnarray*}
S_{n+1} \,\, &\leq& \,\, S_n^{-1/2} \,\,\left( \sum_{j=1}^m \,\, w_j \, 
S_n^{1/2} A_j \, S_n^{1/2} \right) S_n^{-1/2}.\\
&=& \sum_{j=1}^m \,\, w_j \, A_j.
\end{eqnarray*}
Thus the sequence $\{S_n\}$ is a bounded sequence in $\mathbb{P}(n).$ Hence it 
has a subsequence converging to a limit $S.$ By the variance inequality 
\eqref{eq78}, $V(S_n) \geq V(S_{n+1})$ for all $n.$ So $\{V(S_n)\}$ is a 
decreasing sequence of positive numbers. Hence it converges. We must have $\lim 
\,\, V(S_n) = V(S).$ Since $K$ is a continuous function, this implies $\lim \,\,
V(K(S_n)) = V(K(S)).$ But $K(S_n) =S_{n+1}.$ So, $ V(K(S)) = V(S).$ Hence, 
using the variance inequality \eqref{eq78}, we have $d^2(S, K(S)) = 0.$ This 
means 
$S=K(S).$ From the definition of $K(S)$ in \eqref{eq74}, this is possible if 
and only if $S=\Omega (w; \, A_1, \ldots, A_m).$ This proves part (i).\\

By the definition of $H(A)$ in \eqref{eq73} we have for every $A$
\begin{eqnarray*}
 A^{1/2} \,\, H(A) A^{1/2} &=& \sum_{j=1}^m \,\, w_j \,\, A^{1/2} \,\, (A^{-1} 
\# A_j) A^{1/2}\\
&=& \sum_{j=1}^m \,\, w_j \,\, (A^{1/2} \,\, A_j \, A^{1/2})^{1/2},
\end{eqnarray*}
and hence,
$$ \sum_{j=1}^m \,\, w_j \,\,d^2(A,A_j) \,\, = \,\, \tr A \,\, + \,\,   
\sum_{j=1}^m \,\, w_j  \tr \, A_j \,\, - \,\, 2\tr A^{1/2} 
\,H(A) A^{1/2}.$$
From this we can see that 
\begin{eqnarray}
V(S_n)  - V(S_{n+1})  =  \tr \,\, S_n - \tr \,\, S_{n+1} &-& 
2\tr \,\, S_n^{1/2} \,\, H(S_n) S_n^{1/2} \nonumber\\
&+& 2\tr \,\, S_{n+1}^{1/2} \,\, H(S_{n+1}) S_{n+1}^{1/2}, \label{eq85}
\end{eqnarray}
and
\begin{equation}
d^2 (S_n, S_{n+1}) =  \tr \,\, S_n + \tr \,\, S_{n+1} - 
2\tr \,\, S_n^{1/2} \,\, H(S_n) S_n^{1/2}. \label{eq86}
\end{equation}
The variance inequality \eqref{eq78} together with these two relations gives
\begin{equation}
\tr \,\, S_{n+1} \,\, \leq \,\, \tr \,\, S_{n+1}^{1/2} \,\, H 
(S_{n+1}^{1/2})\,\, S_{n+1}^{1/2}. \label{eq87}
\end{equation}
From \eqref{eq86} and \eqref{eq87} we obtain
\begin{eqnarray}
 0 &\leq& d^2 \,\, (S_{n+1}, S_{n+2}) \nonumber \\ 
&=& \tr \,\, S_{n+1} + \tr \,\, S_{n+2} \,\ - \,\, 2 \tr \,\, S_{n+1}^{1/2} H 
(S_{n+1})\,\, S_{n+1}^{1/2} \nonumber \\
&\leq& \tr \,\,S_{n+2} \,\, - \,\,  \tr \,\, S_{n+1}^{1/2} \,\, H (S_{n+1})\,\, 
S_{n+1}^{1/2}. \label{eq88}
\end{eqnarray}
Finally, from \eqref{eq87} and \eqref{eq88} we see that 
$$ \tr \,\, S_{n+1} \leq \tr \,\, S_{n+2}.$$
That proves (ii).
\end{proof}
\section{remarks}
The geometric mean $A\#B$ has played a crucial role at several places in this 
paper. This is the midpoint of the geodesic joining $A$ and $B$ with the 
Riemannian metric $\delta$ defined in \eqref{eq43} and \eqref{eq44}. The 
barycentre of $m$ matrices $A_1, \ldots, A_m$ with weights $w_1, \ldots, w_m$ 
with respect to this metric is defined as
$$G (w; A_1, \ldots, A_m) \,\,= \,\, \underset{X > 0}{\argmin} \sum_{j=1}^m 
w_j \delta^2 (X,A_j).$$
This has been an object of intense study in recent years. See \cite{rbh1} 
\cite{rbh3} \cite{bh} \cite{bk} \cite{h} \cite{ll} \cite{lp1} \cite{lp2} 
\cite{m1}. A natural 
question, from the perspective of matrix analysis, would be to find comparisons 
between the two means $G$ and $\Omega.$ \\
Another classical family of means, called the {\it power means} is defined as
$$ Q_t (w; A_1, \ldots, A_m) = \left(\sum_{j=1}^m w_j A_j^t \right)^{1/t}, 
\,\,\, t > 0. $$
These play an important role in analysis. When $t=\frac{1}{2},$ we have
$$ Q_{1/2} (w; A_1, \ldots, A_m) = \left(\sum_{j=1}^m w_j A_j^{1/2} 
\right)^2.$$
In the special case when $A_1, \ldots, A_m$ are commuting matrices, the 
Wasserstein mean $\Omega$ and the mean $Q_{1/2}$ coincide. If we let
$$ \rho (A,B) \,\, = \,\, \|A^{1/2} - B^{1/2} \|_2 \,\, = \,\, \left[ \tr A + 
\tr B - 2\tr A^{1/2} B^{1/2} \right]^{1/2},$$
then
$$ Q_{1/2} (w; A_1, \ldots, A_m) = \underset{X > 0}{\argmin} \,\, \sum_{j=1}^m 
w_j \rho^2 (X, A_j)$$
It is natural to ask for comparisons between the means $ Q_{1/2}$ and 
$\Omega.$\\

These problems are studied in our forthcoming papers.

\vskip.2in
\noindent{\it Acknowledgement.} The work of R. Bhatia is supported by a J. C. 
Bose National Fellowship and of T. Jain by a SERB Women Excellence 
Award. The work of Y. Lim is supported by the National Research Foundation
 of Korea (NRF) grant founded by the Korea government (MEST) (No. 2015R1A3A2031159).

\vskip.2in
\hfill{December 4, 2017}

% \vskip0.3in
% \begin{flushright}
%  Rajendra Bhatia
% \end{flushright}

\end{document}